\documentclass{amsart}

\usepackage{amsmath,amssymb,latexsym,amsthm,graphicx}

\newtheorem{theorem}{Theorem}[section]

\newtheorem{lemma}[theorem]{Lemma}

\newtheorem{prop}[theorem]{Proposition}

\newtheorem{defi}[theorem]{Definition}

\newcommand{\mR}{\mathcal{R}}

\newcommand{\mD}{\mathcal{D}}

\newcommand{\mQ}{\mathcal{Q}}

\newcommand{\mK}{\mathcal{K}}

\newcommand{\mC}{\mathcal{C}}

\newcommand{\rN}{\mathbb{R}}

\newcommand{\Hs}{\mathbb{H}}

\newcommand{\cN}{\mathbb{C}}

\newcommand{\nN}{\mathbb{N}}

\newcommand{\uS}{\mathbb{S}}

\newcommand{\uB}{\mathbb{B}}

\newcommand{\dB}{B}

\newcommand{\dS}{S}

\newcommand{\llg}{\lambda}

\newcommand{\ag}{\alpha}

\newcommand{\bg}{\beta}

\newcommand{\sg}{\sigma}

\newcommand{\og}{\omega}

\title[Range description for a spherical mean transform]{Range description for a spherical mean transform on spaces of constant curvatures}

\author{Linh V. Nguyen}
\address{Department of Mathematics, University of Idaho, Moscow, Idaho 83843}

\begin{document}

\maketitle

\begin{abstract}
We describe the range of a restricted spherical mean transform, which sends a function supported inside a closed ball in a hyperbolic space to its mean values on the geodesics spheres centered at the boundary of the ball. The description resembles that of the same transform on the Euclidean spaces obtained  by Mark Agranovsky, David Finch, and Peter Kuchment [Inverse Problems and Imaging, 3(3):373--382, 2009] and Mark Agranovsky and Linh V. Nguyen [J. Anal. Math., 112:351--367, 2010].

We also derive a similar characterization for the corresponding transform on the two dimensional spherical space.


\end{abstract}

\section{Introduction}\label{S:Intro}
Let us briefly introduce some basic notions of the hyperbolic space $\Hs^n$. We will work with its unit ball model, which is $\uB^{n} \equiv \{x \in \rN^n: |x|< 1\} \subset \rN^n$ equipped with the metric: $$ds^2(x)= \frac{4}{(1-|x|^2)^2} dx^2.$$ It is a Riemannian manifold of constant sectional curvature $-1$. The distance $d=d_{\Hs^n}(x,y)$ between two points $x$ and $y$ is: $$d_{\Hs^n}(x,y) = \log \left(\frac{\sqrt{1-2 \left<x,y \right>+ |x|^2 |y|^2} +|x-y|}{\sqrt{1-2 \left<x,y \right>+ |x|^2 |y|^2}-|x-y|} \right).$$
The Laplace-Beltrami operator $\Delta$ on $\Hs^n$: $$\Delta =  \frac{1}{4}(1-|x|^2)^{n} \nabla \left((1-|x|^2)^{2-n}\nabla \right).$$
In terms of polar coordinates,\begin{equation} \label{E:Polar}\Delta = \frac{\partial^2}{\partial r^2}+(n-1) \coth(r) \frac{\partial}{\partial r} + \frac{1}{\sinh^2(r)} \Delta_{\uS^{n-1}}.\end{equation} Here, $r=d_{\Hs^n}(x,0)$ and $\Delta_{\uS^{n-1}}$ is the Laplace-Beltrami operator on the Euclidean unit sphere $\uS^{n-1} \subset \rN^n$ applying to the variable $\theta = \frac{x}{|x|}$.

Let us now define the spherical mean transform $\mR$ on $\Hs^n$. For each $x \in \Hs^n$ and $r >0$, let $S_r(x) = \{y \in \Hs^n:d_{\Hs^n}(x,y)=r\}$ be the (non-Euclidean) sphere of radius $r$ centered at $x$. Then, $\mR(f)(x,r)$ is the mean value of $f$ on $S_r(x)$:  $$\mR(f)(x,r)= \frac{1}{|S_r(x)|}\int\limits_{S_r(x)} f(y) d\sg(y).$$ Here, $d\sg(y)$ is the measure on $S_r(x)$ induced by the above Riemannian metric and $|S_r(x)|$ is the total measure of $S_r(x)$. The transform $\mR$ has the following PDE characterization. Let $G(x,r) =\mR(f)(x,r)$, then (e.g., \cite{HelGeo}): \begin{eqnarray}\label{E:Darbh} \left\{ \begin{array}{l} [\partial_r^2 +(n-1)  \coth(r) \partial_r  -\Delta] G(x,r) =0, (x,r) \in \Hs^n \times \rN_+,\\  G(x,0)=f(x),~ G_r(x,0) = 0,~ x \in \Hs^n, \end{array}\right. \end{eqnarray} where, $\rN_+$ is the set of positive real numbers. Conversely, if $G(x,r) \in C^\infty(\Hs^n \times \overline{\rN}_+)$ satisfies the above equation,  then $G(x,r) = \mR(f)(x,r)$. The spherical mean transform on the hyperbolic spaces and other symmetric spaces has attracted considerable attention (e.g., \cite{HelGeo,BerZalc,ABK,Olev}). Let $S \subset \Hs^n$ be a sphere and $\mR_{S}(f)$ be the restriction of $\mR(f)$ to the set of spheres centered at $S$. The analog of $\mR_S$ on the Euclidean spaces plays an important role in thermoacoustic tomography, an emerging biomedical imaging modality (see, e.g., \cite{FPR,FHR,KKun}).


In this article, we study the range description of $\mR_S$. More precisely, we investigate the necessary and sufficient conditions for  a function $g$ defined on $S \times \overline{\rN}_+$, such that $$g(x,r) =\mR_{S}(f)(x,r), \quad (x,r) \in S \times \rN_+,$$ for some function $f \in C^\infty_0(\overline{B})$. Here, $B$ is the ball in $\Hs^n$ whose boundary is $S$ and $C_0^\infty(\overline{B})$ is the space of all functions $f \in C_0^\infty(\Hs^n)$ such that $supp(f) \subset \overline{B}$. 

The same problem in the Euclidean spaces $\rN^n$ has been resolved.  When $n=2$, the description was discovered by Ambartsoumian and Kuchment \cite{AmK06}. It includes the {\bf smoothness \& support}, {\bf orthogonality}, and {\bf moment} conditions. Finch and Rakesh \cite{FR07} obtained the range description for the transform in the odd dimensional spaces. It only includes two conditions, the smoothness \& support and orthogonality conditions. Agranovsky, Kuchment, and Quinto \cite{AKQ} then derived the range description for arbitrary dimensions. It includes three conditions: smoothness \& support, orthogonality, and moment conditions. They also proved that for odd dimensions, the moment condition is not needed. Then, Agranovsky, Finch, and Kuchment  \cite{AFK} proved that for all dimensions, the moment condition follows from the other two. This, indirectly, shows that the {\bf smoothness \& support and orthogonality} conditions are sufficient for the range description in all dimensions. Recently, Agranovsky and the author  took a different approach to prove that the two aforementioned conditions are sufficient for the range description \cite{ANg}. Our proof relies on the global extendibility of the solution for a Darboux-Euler-Poisson problem.

Let us return to the problem in the hyperbolic space. We might assume that  $S$ and $B$ are centered at the origin. That is, $S=S_R(0)$ and $B=B_R(0)$ for some $R>0$. We now describe the necessary conditions. The first condition, which we call {\bf support \& smoothness} condition is quite easy to observe: $g \in C^\infty(S \times \overline{\rN}_+)$, $g(x,r)=0$ for all $r \geq 2R$, and $g$ vanishes up to infinite order at $r=0$. We denote the space of all such functions by $C_0^\infty(S \times[0,2R])$.  The second one, orthogonality condition, comes from the PDE characterization (\ref{E:Darbh}) of the spherical mean transform.  To fully understand that condition, let us make a detour to some basic harmonic analysis on hyperbolic spaces. 

We recall that a horosphere in $\Hs^n$ is a Euclidean sphere contained in $\overline{\uB^n}$ and tangent to $\uS^{n-1}$ at a unique point $\eta$. A horosphere in a hyperbolic space plays the same role as a hyperplane in an Euclidean space. The tangent point $\eta$ determines the "normal direction" of the horosphere. Let $x \in \Hs^{n} \equiv \uB^n$ and $\eta \in \uS^{n-1}$, we define:
$$\left<x,\eta \right>= \log \left(\frac{1-|x|^2}{|x-\eta|^2}\right),$$ where $|x|$ is the Euclidean norm of $x \in \uB^n$. 
This is the distance, determined by the Riemannian metric on $\Hs^n$, from the origin $0 \in \uB^{n}$ to the horosphere through $x$ tangent to $\uS^{n-1}$ at $\eta$. For each $\llg \in \rN$, let $\mu=\mu(\llg)= \frac{2 i\llg+(n-1)}{2}$. Then,
\begin{equation} \label{E:Eig} \Delta_x e^{ \mu \left<x,\eta \right>} = -\frac{(n-1)^2+ 4 \llg^2}{4} e^{\mu \left<x,\eta \right>}.\end{equation}
Let $Y^m$ be a spherical harmonic of degree $m$ on $\uS^{n-1}$. That is, $Y^m$ satisfies:   $$\Delta_{\uS^{n-1}} Y^m(\eta)=-m(m+n-2)Y^m(\eta).$$
We define the function \begin{eqnarray}\label{E:Phim} \Phi_{m,\llg}(x) = \frac{1}{\og_{n-1}}\int\limits_{\uS^{n-1}} e^{ \mu \left<x,\eta\right>} Y^m(\eta) d\eta,\end{eqnarray} where $d\eta$ is the measure on $\uS^{n-1}$ induced by the Euclidean metric in $\rN^n$. 
Then, there is a function $h_{m,\llg}:\overline{\rN}_+ \to \rN$ such that (e.g., \cite{Volch-b}): $$ \Phi_{m,\llg}(x)=h_{m,\llg}(r) Y^m(\theta).$$ Here, $r=d_{\Hs^n}(x,0)$ and $\theta=\frac{x}{|x|}$. Due to (\ref{E:Eig}): $$\Delta \Phi_{m,\llg}(x) = -\frac{(n-1)^2+ 4\llg^2}{4} \Phi_{m,\llg} (x). $$
From the polar coordinate decomposition (\ref{E:Polar}), we obtain: $$\left[d^2_r + (n-1) \coth(r) d_r- \frac{m(m+n-2)}{\sinh^2(r)}\right] h_{m,\llg}(r) = - \frac{(n-1)^2+4 \llg^2}{4} h_{m,\llg}(r).$$
Due to (\ref{E:Phim}), we also observe that $h_{m,\llg}(0)=1$ and $\frac{d h_{m,\llg}}{dr}(0)=0$. In particular, $h_\llg:=h_{0,\llg}$ satisfies: \begin{equation}\label{E:h} \left\{\begin{array}{l} \left[d^2_r + (n-1) \coth(r) d_r \right] h_\llg(r) = - \frac{(n-1)^2+ 4\llg^2}{4} h_\llg(r), \\h_{\llg}(0)=1,~ h'_\llg(0)=0.\end{array}\right.\end{equation}
The function $h_\llg(r)$ is the hyperbolic analog of the well known Bessel function $j_{\frac{n}{2}-1}(\llg r)$, which was used to characterized the spherical transform $\mR_S$ on the Euclidean spaces (e.g., \cite{AKQ}). 

Let us consider the Laplace-Beltrami operator $\Delta$ in $\dB$ with zero Dirichlet boundary condition. It has a discrete set of eigenvalues $\left\{-\frac{(n-1)^2+ 4 \llg_k^2}{4} \right\}_{k=1}^\infty$.  Let $\{\varphi_k\}_{k=1}^\infty$ be an orthonormal basis of $L^2(B)$ consisting of the corresponding eigenfunctions. For each $k$, we denote $u_k(x,r) = h_{\llg_k}(r) \varphi_k(x)$. Then, $u_k \in C^\infty(\overline{B} \times \overline{\rN}_+)$ and satisfies the equation: \begin{eqnarray}\label{E:InDarb} \left\{ \begin{array}{l} [\partial_r^2 +(n-1)  \coth(r)\partial_r -\Delta] u_k(x,r) =0,~ (x,r) \in \dB \times \rN_+, \\ \partial_r u_k(x,0)=0,~ \forall x \in \dB,\\ u_k(y,r)=0,~\forall (y,r) \in \dS \times \rN_+. \end{array} \right. \end{eqnarray}

We are now ready to describe the orthogonality condition for $\mR_S$. Multiplying the equation (\ref{E:Darbh}) by $\sinh^{n-1}(r) u_k(x,r)$ and then taking integration over the domain $\dB \times \rN_+$, we obtain \begin{eqnarray*} \int\limits_{\rN_+} \int\limits_{\dB}[\partial^2_{r} + (n-1) \coth(r) \partial_r - \Delta] G(x,r) u_k(x,r) \sinh^{n-1}(r) dx dr =0.\end{eqnarray*}
Taking integration by parts and using equation (\ref{E:InDarb}), we arrive to \begin{eqnarray*}\int\limits_{\rN_+} \int\limits_{\dS} g(x,r) \partial_\nu u_k(x,r) \sinh^{n-1}(r) d\sg(x) dr =0.\end{eqnarray*}
Here, $\partial_\nu$ is the outward normal derivative. Since $u_k(x,r)= h_{\llg_k}(r) \varphi_k(x)$, we obtain the {\bf orthogonality condition}: $$\int \limits_{\rN_+} \int\limits_{\dS} g(x,r) \partial_\nu \varphi_k(x) h_{\llg_k}(r) \sinh^{n-1}(r) d\sg(x) dr =0.$$

In this article, we prove that the {smoothness \& support and orthogonality} conditions are sufficient for range the description of $\mR_{\dS}$: \begin{theorem}\label{T:MainH}
Let $g$ be a function defined on $\dS \times \overline{\rN}_+$. Then, there is a function $f \in C_0^\infty(\overline{\dB})$ such that $g=\mR_{\dS}(f)$ if and only if the following conditions are satisfied:
\begin{itemize}
\item[1.]{Smoothness $\&$ support condition:} $g \in C_0^\infty(\dS \times [0,2R])$.
\item[2.]{Orthogonality condition:} \begin{eqnarray}\label{E:OrthS} \int \limits_{\rN_+} \int\limits_{\dS} g(x,r) \partial_\nu \varphi_k(x) h_{\llg_k}(r) \sinh^{n-1}(r) d\sg(x) dr =0.\end{eqnarray}
\end{itemize} 
\end{theorem}

A similar result also holds for the same transform on the spherical space $\uS^2$, which is a unit sphere in $\rN^{3}$ (see, Theorem \ref{T:MainP}). The article is organized as follows. Sections \ref{S:Proof} and \ref{S:LM} are dedicated to the proof of Theorem \ref{T:MainH}. The statement and a sketchy proof of the result for the transform on $\uS^2$ is presented in Section \ref{S:P}.

\section{Proof of theorem \ref{T:MainH}}\label{S:Proof}
The necessity of the conditions in Theorem \ref{T:MainH} was discussed in Section \ref{S:Intro}.  We now prove the sufficiency, that is, under assumption 1) and 2) of Theorem \ref{T:MainH}, there is a function $f \in C_0^\infty(\overline{\dB})$ such that $\mR_S(f)=g$. 
We follow the strategy in \cite{AKQ} and \cite{ANg}.  The proof consists of three steps. The first one is to show that there exists a (unique) solution $U \in C^\infty(\overline{B}  \times \overline{\rN}_+)$ of the time-reversed problem \begin{eqnarray} \label{E:TimeR}\left\{ \begin{array}{l} [\partial_r^2 +(n-1)  \coth(r) \partial_r -\Delta] U(x,r)=0, \mbox{ in } \dB \times (0,2R], \\ U(x,r)|_{\dS\times[0,2R]} =g(x,r),\quad (x, r) \in \dS \times (0,2R].\\
U(x,r)=0,~ U_r(x,r)=0,~\forall x \in \dB \mbox{ and } r \geq 2R.
\end{array} \right.
\end{eqnarray}
The second step is to prove that $f^*(x)=U(x,0)$ vanishes up to infinite order on the boundary $\dS$ of $\dB$. The last one is to  show that $g=\mR_\dS(f)$, where $f$ is the zero extension of $f^*$ to $\Hs^n$.

\subsection{Regularity of the internal solution}

\begin{theorem}\label{T:ST1} Assume the smoothness $\&$ support and orthogonality conditions. Then, problem (\ref{E:TimeR}) has a unique solution $U \in C^\infty( \overline{\dB} \times \overline{\rN}_+)$. \end{theorem}
Equation (\ref{E:TimeR}) is hyperbolic for all $r>0$. Standard theory for hyperbolic equation (e.g., \cite{Evb}) shows that $U \in C^\infty(\overline{\dB} \times \rN_+)$. However, when $r=0$ the equation is singular. This prevents us from using any standard theory to derive the regularity of the solution at $r=0$. Our strategy is to get rid of that singularity by transforming (\ref{E:TimeR}) into a wave equation.

Given a Hilbert space $H$, we now introduce some notations. The space $C_0^\infty(\rN;H)$ consists of smooth functions $v: \rN \to H$ with compact supports, and $\mC(H)$ consists of even fuctions $v \in C_0^\infty(\rN;H)$. Given $a>0$, $\mC_a(H)=\{v \in \mC(H): supp(f) \subset [-a,a]\}$. 
\begin{defi}  Let $PW_a(H)$ be the space of entire functions $g: \cN \to H$ such that for each $k \in \nN$, there is a constant $C_k$ such that $$\|g(\llg)\| \leq C (1+|\llg|)^{-N} e^{a \Im (\llg)},~\forall \llg \in \cN.$$\end{defi}

In this article, $H$ will be either $\cN,H^s(S)$, or $H^s(B)$ (for some $s>0$), depending on the context. Let $v \in \mC(H)$, we define the following Fourier-Legendre transform $\widehat{v}$ of $v$ by: $$\widehat{v}(\llg)=\int\limits_0\limits^\infty v(r) h_\llg(r) \sinh^{n-1}(r) dr,$$ where $h_\llg(r)$ is defined in equation (\ref{E:h}). Similar to the standard Fourier transform, this transform also has the Paley-Wiener property:

\begin{lemma}\label{L:PW}
An function $g: \rN \to H$ satisfies $g=\widehat{v}$ for a unique $v \in \mC_a(H)$ if and only if:
\begin{enumerate}
\item $g$ is even,
\item $g$ extends to a function in $PW_a(H)$.
\end{enumerate}
\end{lemma}
The Lemma is due to \cite{HelDual} for $H=\cN$. The general case can be proved exactly in the same way. 
\begin{prop}
Let $H$ be a Hilbert space, and $a>0$. For each $u=u(t) \in \mC_a(H)$, let $T(u)(r)=v(r)$ be such that $\widehat{v}(\llg) = \widetilde{u}(\llg)$, where $\widetilde{u}$ is the Fourier transform of $u$. Then, $T: \mC_a(H) \to \mC_a(H)$ is bijective and satisfies \begin{equation} \label{E:TH} T \left(u_{tt}-\frac{(n-1)^2}{4}u\right) =B_rT(u),\end{equation} where $$B_r=\frac{d^2}{d r^2}+ (n-1) \coth(r) \frac{d}{d r}.$$
$\fbox{}$
\end{prop}
 
 \begin{proof} 
We first show that $T$ is a well defined map from $\mC_a(H)$ to itself. Let $u \in \mC_a(H)$, then $\widetilde{u}$ is extends to an even function in $PW_a(H)$, due to the standard Paley-Wiener theory for the Fourier transform. Due to  Lemma \ref{L:PW}, there is a unique function $v \in \mC_a(H)$ such that $\widehat{v}=\widetilde{u}$. Therefore, $T: \mC_a(H) \to \mC_a(H)$ is well defined. The bijectivity is straight forward.

We now prove (\ref{E:TH}). Indeed, the Fourier transform of $\left(u_{tt}-\frac{(n-1)^2}{4}u\right)$ is equal to \begin{equation} \label{E:Furr} -\frac{(n-1)^2+4\llg^2}{4}\widetilde{u}(\llg).\end{equation} 
Due to (\ref{E:h}), $$B_r(h_\llg) = -\frac{(n-1)^2+4 \llg^2}{4} h_\llg.$$ Simple integrations by parts show that the Fourier-Legendre transform of $B_rv$ is equal to: \begin{equation}\label{E:Lb}-\frac{(n-1)^2+ 4 \llg^2}{4} \widehat{v}(\llg).\end{equation}
If $v=T(u)$, then (\ref{E:Furr}) and (\ref{E:Lb}) are equal. This implies: 
\begin{eqnarray*} T\left(u_{tt}-\frac{(n-1)^2}{4}u\right) =B_r v =  B_r T(u).\end{eqnarray*}
 \end{proof}
 
\begin{proof}[Proof of Theorem \ref{T:ST1}]  Let us extend $g(x,r)$ to an even function in $r$. The orthogonality condition (\ref{E:OrthS}) is equivalent to:
\begin{eqnarray}\label{E:OBF}  \int\limits_{\dS} \widehat{g}(x,\llg_k) \partial_\nu \varphi_k(x) d\sg(x) =0.\end{eqnarray}
Here, $\widehat{g}(x,\llg)$ is the Fourier-Legendre transform of $g(x,r)$ with respect to the variable $r$. 

To explain the idea, let us at the moment assume that $U \in C^\infty(\overline{\dB} \times \overline{\rN}_+)$. Extending $U$ evenly with respect to $r$, we obtain $U \in \mC_{2R}(H^s(\dB))$ for any $s>0$. Moreover, due to and (\ref{E:TimeR}) and (\ref{E:TH}), $V=T^{-1}(U) \in \mC_{2R}(H^s(B))$ solves: \begin{eqnarray}\label{E:CWave} \left\{\begin{array}{l} V_{tt}(x,t)-\frac{(n-1)^2}{4}V(x,t) - \Delta V(x,t) =0,~ x \in \dB, \\ V(x,t)=b(x,t),~ x \in \dS,\\  V(x,-2R)=0,~V_t(x,-2R)=0, ~ x \in \dB \end{array}  \right. \end{eqnarray}
where $b(x,t) =T^{-1}(g)$. Conversely, given a solution $V \in \mC_{2R}(H^s(B))$ of (\ref{E:CWave}), $U=T(V) \in \mC_{2R}(H^s(\dB))$ is the solution of (\ref{E:TimeR}). Therefore, it suffices to prove that (\ref{E:CWave}) has a solution $V \in \mC_{2R}(H^s(\dB))$ for all $s>0$. We now proceed to achieve this goal.

For any $s>0$, due to the smoothness \& support condition: $g \in \mC_{2R}(H^s(\dS))$. Hence, $b=T^{-1}(g) \in \mC_{2R}(H^s(\dS))$.  Because of the standard theory for wave equations (e.g., \cite{Evb}), equation (\ref{E:CWave}) has a solution $V \in C^\infty([-2R,\infty); H^{s}(\dB))$. Extending it by zero for $t \leq -2R$, we obtain $V \in C^\infty(\rN; H^s(B))$. Let $W=V-E(b)$, where $E(b)(.,t)$ is the harmonic extension of $b(.,t)$ to $\overline{\dB}$. We arrive to \begin{eqnarray}\label{E:WWave} \left\{\begin{array}{l} W_{tt}(x,t) - \frac{(n-1)^2}{4}W(x,t)-\Delta W(x,t) =- P(x,t), (x,t) \in \dB \times \rN, \\ W(x,t)=0,~ x \in \dS, \\ W(x,-2R)=W_t(x,-2R)=0,~ x \in \dB. \end{array}  \right. \end{eqnarray} Here, $$P(x,t) = \partial_t^2 E(b)(x,t)-\frac{(n-1)^2}{4}E(b)(x,t)= E\left(b_{tt}-\frac{(n-1)^2}{4}b\right)(x,t)$$ is the harmonic extension of $b_{tt}-\frac{(n-1)^2}{4}b$. 
We now expand the functions $W$ and $P$ in terms of the orthonormal basis $\{\varphi_k\}_k$:
\begin{eqnarray*}W(x,t) =\sum_{k} \og_k(t) \varphi_{k}(x), \quad P(x,t) =  \sum_{k} p_k(t) \varphi_{k}(x).\end{eqnarray*} 
Equation (\ref{E:WWave}) reduces to  \begin{eqnarray*}\left\{ \begin{array}{l}  \og''_k(t)-\frac{(n-1)^2}{4}\og_k(t) + \frac{(n-1)^2+ 4 \llg_k^2}{4} \og_k(t) = - p_k(t), \\ \og_k(-2R)=\og'_k(-2R)=0. \end{array} \right.\end{eqnarray*} Here, $-\frac{(n-1)^2+4 \llg^2_k}{4}$ is the eigenvalue corresponding to $\varphi_k$. The above equation is equivalent to: 
\begin{eqnarray}\label{E:SODE} \left\{ \begin{array}{l}  \og''_k(t) + \llg_k^2 \og_k(t) = - p_k(t), \\ \og_k(-2R)=\og'_k(-2R)=0. \end{array} \right.\end{eqnarray}
We notice that the function $p_k$ is determined by $$p_k(t) = \int\limits_B P(x,t) \varphi_k(x)=\frac{-4}{(n-1)^2+ 4 \llg_k^2} \int \limits_B P(x,t) \Delta \varphi_k(x) dx .$$
Taking integration by parts, we obtain \begin{eqnarray*} p_k(t) = \frac{-4}{(n-1)^2+ 4 \llg_k^2} \int\limits_S P(x,t) \partial_\nu \varphi_k(x) d\sg(x).\end{eqnarray*}
Since $P(x,t)$ is the (harmonic) extension of $b_{tt}(x,t)-\frac{(n-1)^2}{4}b$:  \begin{eqnarray*} p_k(t) =  \frac{-4}{(n-1)^2+ 4 \llg_k^2} \int\limits_{\dS} \left(b_{tt}(x,t)-\frac{(n-1)^2}{4} b(x,t)\right) \partial_\nu \varphi_k(x) d\sg(x).\end{eqnarray*}
Taking the Fourier transform, we derive  \begin{eqnarray*} \widetilde{p}_k(\llg) =  \frac{(n-1)^2+ 4 \llg^2}{(n-1)^2+ 4 \llg_k^2} \int\limits_{\dS} \widetilde{b}(x,\llg) \partial_\nu \varphi_k(x) d\sg(x).\end{eqnarray*}
Since $g=T(b)$, we arrive to \begin{eqnarray*} \widetilde{p}_k(\llg) =  \frac{(n-1)^2+ 4 \llg^2}{(n-1)^2+ 4 \llg_k^2} \int\limits_{\dS} \widehat{g}(x,\llg) \partial_\nu \varphi_k(x) d\sg(x).\end{eqnarray*}
Due to (\ref{E:OBF}), $\widetilde{p}_k(\llg_k)=0$. Therefore, $$\frac{\widetilde{p}_k(\llg)}{\llg^2-\llg^2_k} \in PW_{2R}(\cN). $$
 Moreover, it is an even function (since $\widetilde{p}_k$ is). Due to the Paley-Wiener theory, there is a function $\gamma \in \mC_{2R}(\cN)$ such that $$\widetilde{\gamma}(\llg)= \frac{\widetilde{p}_k(\llg)}{\llg^2-\llg^2_k}.$$
Taking the inverse Fourier transform, we observe that $\gamma$ solves the equation (\ref{E:SODE}). Since equation (\ref{E:SODE}) has a unique solution, one obtains $\og_k =\gamma$. This implies $\og_k \in \mC_{2R}(\cN)$ for all $k$. Incorporating this with the fact that $W \in C^\infty(\rN; H^s(B))$, we obtain $W \in C_{2R}(H^s(B))$ for any $s>0$. Recalling that $b \in \mC_{2R}(H^s(S))$, we derive $E(b) \in \mC_{2R}(H^s(B))$. Therefore, $V=W+E(b) \in \mC_{2R} (H^s(B))$. This finishes our proof.
\end{proof}

\subsection{Vanishing at the boundary}
Let $U \in C^\infty(\overline{\dB} \times \overline{\rN}_+)$ be the solution of equation (\ref{E:TimeR}) obtained in the previous section. Let $f^*=U(.,0)$, then: 
\begin{eqnarray} \label{E:FW}\left\{ \begin{array}{l} [\partial_r^2 + (n-1)\coth(r) \partial_r - \Delta] U(x,r)=0,~(x,t) \in \overline{\dB} \times \overline{\rN}_+, \\ U(x,r)=g(x,r),~ (x, r) \in \dS \times \overline{\rN}_+,\\ U(x,0)=f^*(x),~ U_r(x,0)=0,~x \in \overline{\dB},
\end{array} \right.
\end{eqnarray} 
We also recall here that $U(x,r)=0$ for $r \geq 2R$. 
We now prove that:
\begin{theorem} \label{T:Sup}
The function $f^*$ vanishes up to infinite order on $\dS$.
\end{theorem}
Let us look at the spherical harmonics expansion of $U$ and $f^*$: 
$$U(x,r) = \sum_{m=0}^\infty \sum_{i=1}^{i_m} u_{m,i}(|x|,r)|x|^m Y^m_i(\theta),~ f^*(x) = \sum_{m=0}^\infty \sum_{i=1}^{i_m} f_{m,i}(|x|)|x|^m Y^m_i(\theta),$$ where $\theta=\frac{x}{|x|}$ and $\{Y_i^m\}_{i=1}^{i_m}$ is an orthonormal basis of spherical harmonics of degree $m$. Here, $u_{m,i}(.,r),f_{m,i}$ extend to smooth even function on $\rN$.
Let $s=d_{\Hs^n}(x,0)$, direct calculation shows $|x| = \tanh\left(\frac{s}{2}\right)$. Setting \begin{eqnarray} \label{E:UF} \left. \begin{array}{l} U_{m,i}(s) = u_{m,i}(|x|,r)|x|^m = u_{m,i}\left(\tanh\left(\frac{s}{2}\right),r\right) \left[\tanh\left(\frac{s}{2}\right)\right]^m, \\F_{m,i}(s) = f_{m,i}(|x|)|x|^m = f_{m,i}\left(\tanh\left(\frac{s}{2}\right)\right)\left[ \tanh\left(\frac{s}{2}\right)\right]^m, \end{array}\right. \end{eqnarray} we arrive to
$$U(x,r) = \sum_{m=0}^\infty \sum_{i=1}^{i_m} U_{m,i}(s,r) Y^m_i(\theta),~ f^*(x) = \sum_{m=0}^\infty \sum_{i=1}^{i_m} F_{m,i}(s) Y^m_i(\theta).$$
It suffices to prove that $F_{m,i}$ vanishes up to infinite order at $s=R$ for all $m,i$. For the sake of simplicity, we will drop the irrelevant index $i$ of $U_{m,i}$ and $F_{m,i}$, when not needed. Let us consider the operator $$\mD_{m,s} = d_s^2 +(n-1) \coth(s) d_s - \frac{m(m+n-2)}{\sinh^2 s}.$$ We will also drop the second index of $\mD_{m,s}$ when there is no ambiguity about the variable.  Due to the polar coordinate formula (\ref{E:Polar}) of $\Delta$, we obtain $$\Delta [\ag(s) Y^m(\theta)] = (\mD_m \ag)(s) Y^m(\theta).$$
Here, as above, $s=d_{\Hs^n}(x,0)$ and $\theta=\frac{x}{|x|}$.
From equation (\ref{E:FW}), we deduce that $U_m(s,r)$ satisfies:
\begin{eqnarray} \label{E:FWdec} \left\{ \begin{array}{l} (\mD_{0,r} - \mD_{m,s}) U_m(s,r)=0,~ (s,r) \in  [0,R] \times \rN_+, \\ U_m(R,r) =g_m(r),~r \in \rN_+,\\
U_m(s,0)=F_m(s),~ \partial_r U_m(s,0)=0,~s \in [0,R]. \end{array} \right.
\end{eqnarray}

\begin{lemma}\label{L:D_m}
For any $l \geq 0$, one has \begin{equation}\label{E:D_m}\left[\mD_m^l F_m\right](R) =0.\end{equation}
\end{lemma}
\begin{proof}
By iterating equation (\ref{E:FWdec}) $l$ times, we arrive to $$\mD^l_{m,s}U_m(s,r) = \mD^l_{0,r} U_m(s,r).$$
At $(s,r)=(R,0)$, we obtain $$\mD^l_{m}F_m(R) = \mD^l_{0} g_m(0).$$
Since $g_m(r)$ vanishes up to infinite order at $r=0$, the above equation gives $\mD^l_{m} F_m(R)=0$.
\end{proof}

\begin{lemma}\label{L:Q_m} Let $\Gamma_{k} = d_s + (n+k-2) \coth(s)$ and $\mQ_m = \Gamma_1....\Gamma_m.$ Then, \begin{equation} \label{E:Q_m}\left[d_s^l \mQ_m F_m\right](R) =0, \quad \forall~ l \geq 0.\end{equation}
\end{lemma}
An analog of the above identity for spherical mean transform on Euclidean spaces was derived in \cite{ANg} (see \cite[Lemma 4.1, i)]{ANg}). The derivation relied on a projection formula,  obtained in \cite{Ep}, for spherical harmonics. We present here a different approach to prove Lemma \ref{L:Q_m}. The idea is to apply an operator to (\ref{E:FWdec}) so that we obtain a symmetric equation, which was investigated in \cite{HelGeo}. Let us state the following identity, whose proof is provided in Appendix:

\begin{prop} \label{P:Com} 
We have the following identity \begin{eqnarray*} \Gamma_k \mD_k = \mD_{k-1}\Gamma_k.
\end{eqnarray*}
\end{prop}

\begin{proof}[Proof of Lemma \ref{L:Q_m}]
Due to equation (\ref{E:FW}) and the fact that $U(x,r)=U_r(x,r)=0$ for all $x\in \overline{B}$ and $r \geq 2R$, the domain of dependence argument implies $$U(x,r)=0, \mbox{ for all $(x,r)$ such that } r-d_{\Hs^n}(x,0) \geq R.$$ This gives \begin{equation}\label{E:0R} U_m(s,r)=0,\mbox{ for all }(r,s) \mbox{ such that } r-s \geq R.\end{equation}
We recall from equation (\ref{E:FWdec}):  $$\mD_{m,s}U_m(s,r) = \mD_{0,r} U_m(s,r).$$
Applying $\Gamma_m$ (with respect to variable $s$) to this equation, we obtain $$\Gamma_m \mD_{m,s}U_m(s,r) = \Gamma_m \mD_{0,r} U_m(s,r).$$
Due to Proposition \ref{P:Com}, we arrive to  $$\mD_{m-1,s} \Gamma_mU_m(s,r) =  \mD_{0,r} \Gamma_m U_m(s,r).$$
Therefore, applying $\Gamma_{m-1},..,\Gamma_{1}$ and using the same argument as above, we obtain \begin{eqnarray*} \mD_{0,s} [\mQ_m U_m](s,r) =  \mD_{0,r} [\mQ_m U_m](s,r).\end{eqnarray*}
Recalling that $\mD_{0,s}= \frac{d^2}{ds^2}+(n-1) \coth(s) \frac{d}{ds}$, the above symmetric equation for $[\mQ_m U_m](s,r)$ was investigated in \cite{HelGeo}. Although the operator $\mQ_m$ is singular at $s=0$, due to the formula (\ref{E:UF}) of $U_m=U_{m,i}$, the function $[\mQ_m U_m] (s,r)$ is smooth on $[0,R] \times \overline{\rN}_+$. From \cite[p. 320-322]{HelGeo}, we obtain \begin{equation}\label{E:mQF}  [\mQ_m U_m](0,s)=[\mQ_m U_m](s,0)=\mQ_m F_m(s), ~ s \in [0,R].\end{equation}
Due to (\ref{E:0R}) and the fact that $\mQ_m$ is a local operator, we obtain $$[\mQ_m U_m](s,r) = 0, \mbox{ for all }(s,t) \mbox{ such that } r-s \geq R.$$
This in particular implies $d_s^l [\mQ_m U_m] (0,s)|_{s=R}=0$. Combining this and (\ref{E:mQF}), we conclude:
$$[d_s^l \mQ_m F_m](R) =0.$$ This finishes the proof.
\end{proof}

The main ingredient for proof of Theorem \ref{T:Sup} is the following lemma: 
\begin{lemma}\label{L:Ind}
Let $F_m$ be a function defined on a neighborhood of $s=R$ such that for all $l=0,..,m-1$:  $$[d_s^l \mQ_m F_m](R) = [\mD_m^l F_m](R) =0. $$ Then, $F_m^{(k)}(R)=0$, for all $k=0,..,2m-1$.
\end{lemma}
The proof of this lemma is the most difficult part of this paper. It will be presented in Section \ref{S:LM}. Here we use it to prove Theorem \ref{T:Sup}. 
\begin{proof}[Proof of Theorem \ref{T:Sup}] Due to Lemmas \ref{L:D_m}, \ref{L:Q_m}, and \ref{L:Ind}, we obtain  $F_m^{(k)}(R)=0$ for all $k=0,1,..,2m-1$. 
From (\ref{E:Q_m}), we deduce for any $l \geq 0$: $$F_m^{(m+l)}(R)+ \sum_{j=0}^{m+l-1} A_{l,j} F_m^{(j)} (R) =0, $$ for some coefficients $A_{l,j}$ independent of $F_m$.
Choosing $l=m,m+1,...$, by induction, we derive $F_m^{(k)}(R) =0$ for all $k \geq 2m$. Therefore, $F_m^{(k)}(R)=0$ for all $k \geq 0$. We recall here that $F_m$ stands for $F_{m,i}$ for all $i=1,..,i_m$. Due to the uniform convergence, we conclude $$f^*(x)=\sum_{m=0}^\infty \sum_{i=1}^{i_m} F_{m,i}(s) Y_i^m(\theta)$$ vanishes up to infinite order at $x \in \dS =S_R(0)$.  
\end{proof}

\subsection{Finishing the proof of Theorem \ref{T:MainH}}
Extending $f^*$ by zero outside $\dB$ to a function $f$, we obtain $f \in C_0^\infty(\Hs^n)$. We now prove that $g=\mR_{\dS}(f)$. Indeed, let $G:= \mR(f)  \in C^\infty(\Hs^n \times \overline{\rN}_+)$. Then, $G$ satisfies \begin{eqnarray*} \left\{ \begin{array}{l} [\partial_r^2 + (n-1) \coth(r) \partial_r - \Delta] G(x,r) =0,~ (x,t) \in \Hs^n \times \rN_+,\\  G(x,0)=f(x),~ G_r(x,0) = 0,~x \in \Hs^n. \end{array}\right. \end{eqnarray*}
We recall that the following equation
\begin{eqnarray*} \left\{ \begin{array}{l} [\partial_r^2 + (n-1)\coth(r) \partial_r - \Delta] U(x,r)=0,~(x,r) \in  \dB \times \rN_+, \\ U(x,r) =g(x,r),~(x, r) \in \dS \times \rN_+,\\
U(x,0)=f^*(x),~ U_r(x,0)=0,~ x \in \dB,
\end{array} \right.
\end{eqnarray*} has a unique solution $U \in C^\infty(\overline{\dB}\times \overline{\rN}_+)$ satisfying $U(x,r) =0$ for $r \geq 2R$.  Setting $H(x,r) = G(x,r) -U(x,r)$, we obtain:
\begin{eqnarray}\label{E:DarD} \left\{ \begin{array}{l} [\partial_r^2 + (n-1) \coth(r) \partial_r - \Delta] H(x,r) =0,~ (x,r) \in \dB \times \rN_+,\\  H(x,0)=0,~ H_r(x,0) = 0,~ x \in \overline{\dB}. \end{array}\right. 
\end{eqnarray}
The domain of dependence argument then shows that $U(x,r)=0$ in the downward cone $$\mK_-= \{(x,r):r \geq 0,~  d_{\Hs^n}(x,0)+r \leq R\}.$$
On the other hand, $f \in C^\infty_0(\overline{\dB})$ implies that $G(x,r)=0$ for all $x \in \overline{\dB}$ and $r \geq 2R$. Hence, $H$ satisfies the time-reversed equation: \begin{eqnarray*} \left\{ \begin{array}{l} [\partial_r^2 + (n-1)\coth(r)\partial_r - \Delta] H(x,r)=0,~ (x,t) \in \dB \times (0,2R],\\
H(x,2R)=0,~ H_r(x,2R)=0,~ x \in \overline{B}.
\end{array} \right.
\end{eqnarray*}
The domain of dependence argument shows that $H(x,r)=0$ in the upward cone $$\mK_+= \{(x,r): r \leq 2, r-d_{\Hs^n}(x,0) \geq R\}.$$
Therefore, $H(x,r)=0$ in $\mK_+ \cup \mK_-$, which implies $\partial^\ag_x H(0,r)=0$ for all $r \in \rN_+$. Applying Fourier-Legendre transform (for variable $r$) to  equation (\ref{E:DarD}), we obtain: $$\frac{(n-1)^2+ 4 \llg^2}{4} \widehat{H}(x,\llg)+ \Delta \widehat{H}(x,\llg) =0,~ x \in \dB.$$ 
The above equation shows that $\widehat{H}(.,\llg)$ is analytic in $\dB$ for any $\llg \in \rN$. Since $\partial^\ag_x H(0,r)=0$, we obtain $\partial^\ag_x \widehat{H}(0,\llg)=0$ for any multi-index $\alpha$. The analyticity of $\widehat{H}(.,\llg)$ then implies $\widehat{H}(x,\llg)=0$ for all $x \in \dB$.  Due to the injectivity of the Fourier-Legendre transform, we obtain $H(x,r)=0$ for all $(x,r) \in \dB \times \overline{\rN}_+$.   Now the continuity of $H$ implies $H(x,r)=0$ for all $\overline{\dB} \times \overline{\rN}_+$. This, in particular, gives $\mR_{\dS}(f)(x,r) = G(x,r) = U(x,r)=g(x,r)$ for all $(x,r) \in \dS \times \overline{\rN}_+$. Theorem \ref{T:MainH} is proved.
\section{Proof of Lemma \ref{L:Ind}} \label{S:LM}
We now prove Lemma \ref{L:Ind}, which is the cornerstone of this article. Let us first state some auxiliary results:
\begin{prop}\label{P:Q_m}
For all $i=0,..,m-1$, let $u_i(s) = \cosh^i(s) \sinh^{-n-m+2}(s)$. Then,  $$\left(\mQ_m u_i\right)(R)=0.$$
\end{prop}
\begin{proof} We have \begin{eqnarray*} \Gamma_k \left[\cosh^i(s) \sinh^{-n-k+2}(s)\right] &=& \left[d_s + (n+k-2)\coth(s) \right]\left[\cosh^i(s) \sinh^{-n-k+2}(s)\right]  \\ &=& i \cosh^{i-1}(s) \sinh^{-n-(k-1)+2}(s).\end{eqnarray*}
Therefore, \begin{eqnarray*} && \mQ_m \left[\cosh^i(s) \sinh^{-k-n+2}(s)\right]=\left[\prod_{j=1}^{m-i-1}\Gamma_j\right] \Gamma_{m-i}...\Gamma_m \left[\cosh^i(s) \sinh^{-n-m+2}(s)\right]  \\ &&= i \left[\prod_{j=1}^{m-i-1}\Gamma_j\right] \Gamma_{m-i}...\Gamma_{m-1}\left[ \cosh^{i-1}(s) \sinh^{-n-(m-1)+2}(s)\right] \\ && = i! \left[\prod_{j=1}^{m-i-1}\Gamma_j\right] \Gamma_{m-i}\left[ \sinh^{-n-(m-i)+2}(s)\right] =0.\end{eqnarray*}
 
\end{proof}
\begin{prop}\label{P:D_m} For all $i=0,..,m-1$:
$$(\mD_m-\kappa_i) u_i(s)= -i(i-1) u_{i-2},$$
where $\kappa_i = (m-i-1) (m+n-2-i)$.
\end{prop}
A proof of this proposition will be provided in Appendix. We now prove Lemma \ref{L:Ind}.
\begin{proof}[Proof of Lemma \ref{L:Ind}] For $l=0,..,m-1$, we can write $$\left[d_s^l \mQ_m F_m\right](R) = \sum_{i=0}^{2m-1} A_{l,i} F_m^{(i)}(R),$$ and $$[\mD_m^l F_m](R) = \sum_{i=0}^{2m-1}B_{l,i}F_m^{(i)}(R).$$ Here $A_{l,i},B_{l,i}$ are constants independent of $F_m$. We denote by $A_l$ and $B_l$ the corresponding row vectors $A_l=(A_{l,i})_{i=0}^{2m-1}$ and $B_l=(B_{l,i})_{i=0}^{2m-1}$.  It suffices to prove that $\{A_l,B_l\}_{l=0}^{m-1}$ is linearly independent. Indeed, assume that $\{\ag_l,\bg_l\}_{l=0}^{m-1}$ is such that \begin{equation} \label{E:Dep} \sum_{l=0}^{m-1} \ag_l A_l + \sum_{l=0}^{m-1} \bg_l B_l =0.\end{equation} We now prove that $\ag_l=\bg_l=0$ for all $l=0,..,m-1$. Let $$P_1(x) = \sum_{l=0}^{m-1}\ag_l x^l,~ P_2(x) = \sum_{l=0}^{m-1} \bg_l x^l.$$
From (\ref{E:Dep}), we obtain $$ \left[P_1(d_s) \mQ_m u\right](R)+ \left[P_2(\mD_m) u\right](R)=0,$$ for any function $u$ smooth at $R$. 
Let $u=u_i$, we obtain $$\left[P_1(d_s) \mQ_m u_i \right](R)+ \left[P_2(\mD_m)u_i \right](R) =0.$$
Due to Propositions  \ref{P:Q_m}, the first term of the left hand side is zero. We, thus, obtain \begin{equation}\label{E:P2Dm} \left[P_2(\mD_m)u_i \right](R) =0,\quad \forall i=0,..,m-1.\end{equation}
We now prove that $$P_2(x)=Q(x) \prod_{i=0}^{m-1}(x-\kappa_i),$$ where $Q$ is a polynomial and $\kappa_i$ is defined in Proposition \ref{P:D_m}. Indeed,
let $i=0$ in equation (\ref{E:P2Dm}).  Due to Proposition \ref{P:D_m}, we deduce $$P_2(\kappa_0) u_0(R) =0.$$
Since $u_0(R) \neq 0$, we obtain $P_2(\kappa_0)=0$. The same argument then shows that $P_2(\kappa_1)=0$.  We now prove, by induction, that $P_2$ is divisible by $(x-\kappa_i)$ for all $2\leq i \leq m-1$. Indeed, assume that it is true for all $k \leq i-1$. We obtain $$P_2(x)=Q(x)\prod_{p=1}^{\lfloor \frac{i}{2}\rfloor}(x-\kappa_{i-2p}).$$
Here, $\lfloor{\frac{i}{2}}\rfloor$ is the the integer part of $\frac{i}{2}$, which is the biggest integer less than or equal to $\frac{i}{2}$. Let us write $Q(x) = q(x) (x-\kappa_i)+C$, where $C$ is a constant. We arrive to $$P_2(x)= q(x) \prod_{p=0}^{\lfloor \frac{i}{2} \rfloor}(x-\kappa_{i-2p})+ C\prod_{p=1}^{\lfloor \frac{i}{2} \rfloor}(x-\kappa_{i-2p}).$$ Therefore, \begin{equation}\label{E:DecP} P_2(\mD_m)u_{i} =  q(\mD_m) \prod_{p=0}^{\lfloor \frac{i}{2} \rfloor}(\mD_m-\kappa_{i-2p})u_{i}+ C\prod_{p=1}^{\lfloor \frac{i}{2} \rfloor}(\mD_m-\kappa_{i-2p})u_{i}.\end{equation}
Due to Proposition \ref{P:D_m}, we deduce \begin{eqnarray*}  \prod_{p=0}^{\lfloor \frac{i}{2} \rfloor}(\mD_m-\kappa_{i-2p})u_{i} &=&  \left[\prod_{p=1}^{\lfloor \frac{i}{2} \rfloor}(\mD_m-\kappa_{i-2p})\right] \left[(\mD_m-\kappa_{i}) u_{i}\right] \\ &=&i(i-1)\left[\prod_{p=1}^{\lfloor \frac{i}{2} \rfloor}(\mD_m-\kappa_{i-2p})\right] u_{i-2}. \end{eqnarray*}
Continuing the argument, we obtain \begin{eqnarray*} \prod_{p=0}^{\lfloor \frac{i}{2} \rfloor}(\mD_m-\kappa_{i-2p})u_{i}= \left\{\begin{array}{l} i! (\mD_m-\kappa_1) u_1, \mbox{ if $i$ is odd}, \\ i!   (\mD_m-\kappa_0) u_0, \mbox{ if $i$ is even}. \end{array} \right.  \end{eqnarray*}
Applying Proposition \ref{P:D_m} once more, we conclude 
\begin{eqnarray*} \prod_{p=0}^{\lfloor \frac{i}{2} \rfloor}(\mD_m-\kappa_{i-2p})u_{i}= 0.\end{eqnarray*}
Hence, equation (\ref{E:DecP}) gives: 
$$P_2(\mD_m)u_i = C\prod_{p=1}^{\lfloor \frac{i}{2} \rfloor} (\mD_m-\kappa_{i-2p})u_{i}= (-1)^{\lfloor \frac{i}{2} \rfloor} C\prod_{p=1}^{\lfloor \frac{i}{2} \rfloor} (\kappa_{i-2p}-\mD_m)u_i.$$
Due to (\ref{E:P2Dm}), we arrive to  \begin{equation} \label{E:Czero} C\left[\prod_{p=1}^{\lfloor \frac{i}{2} \rfloor} (\kappa_{i-2p}-\mD_m)u_i\right](R)=0.\end{equation}
Let us recall that Proposition \ref{P:D_m} gives 
$$(\kappa_j-\mD_m)u_{k} = (\kappa_j -\kappa_k) u_{k} + k(k-1) u_{k-2}.$$
Therefore,
\begin{eqnarray*} \prod_{p=1}^{\lfloor \frac{i}{2} \rfloor}(\kappa_{i-2p}-\mD_m) u_i &=& \left[\prod_{p=2}^{\lfloor \frac{i}{2} \rfloor}(\kappa_{i-2p}-\mD_m)\right] [(\kappa_{(i-2)} -\mD_m) u_i]  \\ &=& \left[\prod_{p=2}^{\lfloor \frac{i}{2} \rfloor}(\kappa_{i-2p}-\mD_m)\right] [(\kappa_{i-2} -\kappa_i) u_{i} + i(i-1) u_{i-2}] \\ &=& (\kappa_{i-2} -\kappa_{i}) \left[\prod_{p=2}^{\lfloor \frac{i}{2} \rfloor}(\kappa_{i-2p}-\mD_m)\right]  u_{i}\\ &+& i(i-1) \left[\prod_{p=2}^{\lfloor \frac{i}{2} \rfloor}(\kappa_{i-2p}-\mD_m)\right]  u_{i-2}. \end{eqnarray*} 
We notice here that $\kappa_i$ is strictly decreasing in $i=0,..,m-1$. Hence, in the last formula, the first coefficient $(\kappa_{i-2}-\kappa_i)$ is positive and the second one $i(i-1)$ is nonnegative.  Continuing the expansion for $p=2,...,\lfloor \frac{i}{2} \rfloor$, we obtain 
\begin{eqnarray*} \prod_{i=1}^{\lfloor \frac{i}{2} \rfloor}(\kappa_{i-2p}-\mD_m) u_i = \sum_{p=0}^{[\frac{i}{2}]} c_{i-2p} u_{i-2p},\end{eqnarray*} where $c_i>0$ and all other coefficients $c_k's$ are nonnegative. Since $u_j(R)>0$ for all $j=0,..,m-1$, we obtain then \begin{eqnarray*}\left[\prod_{i=1}^{\lfloor \frac{i}{2} \rfloor}(\kappa_{i-2p}-\mD_m) u_i \right](R)= \sum_{p=0}^{\lfloor \frac{i}{2}\rfloor} c_{i-2p} u_{i-2p}(R) >0. \end{eqnarray*}  From (\ref{E:Czero}), we then derive $C=0$. Thefore, $P_2$ is divisible by $(x-\kappa_i)$, for $i=0,..,m-1$.
Since $\kappa_i$ are pairwise different, we obtain the factorization $$P_2(x) = Q(x) \prod_{i=0}^{m-1} (x-\kappa_i).$$ We now recall that $P_2$ is a polynomial of degree at most $m-1$. This shows $P_2=0$, and so $\bg_l=0$ for all $i=0,..,m-1$. 
From (\ref{E:Dep}), we arrive at $$\sum_{l=0}^{m-1} \ag_l A_l=0.$$
It is easy to observe that, for each $l \geq0$, $A_{l,m+l}=1$ and $A_{l,j}=0$ for all $j>m+l$. Hence, the above equation gives $\ag_l=0$ for all $l=0,..,m-1$. This shows the linear independence of $\{A_l,B_l\}_{l=0}^{m-1}$. It also finishes the proof of the lemma.
\end{proof}

\section{The result on spherical geometry}\label{S:P}
In this section, we present the range description for the spherical mean transform on $\uS^2$. The argument is essentially the same as that on the hyperbolic space $\Hs^n$. The result for higher dimensional spaces $\uS^n$ ($n \geq 3$) might require more work. We will discuss about it in an upcoming paper.

 We recall that $\uS^2$ is the unit sphere in $\rN^{3}$. It is a Riemannian manifold with the metric induced by the Euclidean one of $\rN^3$. Let $f \in C^\infty(\uS^2)$, for each $x \in \uS^2$ and $0< r < \pi$, we define: $$\mR(f)(x,r)=\frac{1}{|S_r(x)|} \int\limits_{S_r(x)} f(y) d\sg(y),$$
where $S_r(x)$ is the sphere in $\uS^2$ of radius $r$ centered at $x$. It is easy to see that $\mR(f)$ extends smoothly to $(x,r) \in \uS^2 \times [0,\pi]$. Moreover,  $G(x,r)=\mR(f)(x,r)$ satisfies the Darboux-type equation \cite{HelGeo}:
 \begin{eqnarray}\label{E:Darbs} \left\{ \begin{array}{l} [\partial_r^2 + \cot(r) \partial_r - \Delta] G(x,r) =0, (x,r) \in \uS^2 \times (0,\pi),\\  G(x,0)=f(x),~ G_r(x,0) = 0,~ x \in \uS^2. \end{array}\right. \end{eqnarray}
Conversely, if $G(x,r) \in C^\infty(\uS^2 \times [0,\pi])$ satisfies the above equation, then $G(x,r) = \mR(f)(x,r)$. Indeed, let $U(x,r) = G(x,r) - \mR(f)(x,r)$. Then, $U$ satisfies the equation  \begin{eqnarray*}\left\{ \begin{array}{l} [\partial_r^2 +  \cot(r) -\Delta] U(x,r) =0,~ (x,r) \in \uS^2 \times (0,\pi),\\  U(x,0) = 0,~ U_r(x,0) = 0, ~ x \in \uS^2. \end{array}\right. \end{eqnarray*}
Multiplying the equation by $U_r(x,r)$ and taking integration by parts, one obtains:  $$\int\limits_{\uS^2} [ U_{rr} (x,r) U_r(x,r)+ \nabla U(x,r) \nabla U_r(x,r)] dx + \cot(r) \int\limits_{\uS^2} |U_r(x,r)|^2 dx=0.$$
Let $$E(r)= \frac{1}{2}\int\limits_{\uS^2} |U_{r} (x,r)|^2+ |\nabla U(x,r)|^2dx.$$
For any $0<\tau < \pi$, there is a positive constant $c>0$ such that $-\cot(r) \leq c$ for all $0< r <\tau$. The above identity then implies $$\frac{d E(r)}{dr} \leq 2c E(r),~0<r<\tau.$$
Applying the Gronwall's inequality, one arrives to $E(\tau) \leq E(0) e^{2 c\tau}$. Since $E(0)=0$ and $E(\tau) \geq 0$, one obtains $E(\tau)=0$ for all $\tau \in (0,\pi)$. This implies $U(x,r)=G(x,r) - \mR(f)(x,r) =0$ for all $(x,r) \in \uS^2 \times [0,\pi)$. Due to the continuity, $G(x,r)=\mR(f)(x,r)$ for all $(x,t) \in \uS^2 \times [0,\pi]$. 

Let $S$ be any sphere in $\uS^2$ of radius $0<R<\frac{\pi}{2}$ and $\mR_S$ be the restriction of $\mR(f)$ to the set of spheres centered at $S$.  We now describe the function $g=\mR_S(f)$, for some function $f \in C_0^\infty(\overline{B})$. Here, $B$ is the ball enclosed by $S$, $\overline{B}$ is its closure in $\uS^2$, and $C_0^\infty(\overline{B})$ is the space  of functions $f \in C_0^\infty(\uS^2)$ such that $supp(f) \subset \overline{B}$. Simple arguments show that $g \in C^\infty(S \times [0,\pi])$, $g(x,r)=0$ for $r \geq 2R$, and $g(x,r)$ vanishes up to infinite order at $r=0$. We, then, say that $g \in C_0^\infty(S \times [0,2R])$. This is the {\bf support \& smoothness condition}.  

The second condition comes from the PDE characterization (\ref{E:Darbs}). Let us consider the eigenvalue problem  \begin{equation}\label{E:sphp}\left[d^2_r + \cot(r) d_r \right] h(r) + \llg (\llg +1)h(r) =0,~ h(0)=1,~h'(0)=0. \end{equation}
This problem has a unique solution for any $\llg \in \cN$ (see, e.g., \cite{BerZalc}): $h_\llg(r) = P_{\llg}(\cos(r))$. Here, $P_\llg$ is the Legendre function of order $\llg$. 

Let us  also consider the eigenvalue problem \begin{eqnarray*}\left\{ \begin{array}{l} -\Delta \varphi (x) = \llg(\llg+1) \varphi (x),~ x\in B,\\ \varphi(y)=0,~ y \in S. \end{array} \right. \end{eqnarray*}
This equation has nontrivial solutions for a discrete set $\{\llg_k\}_{k=1}^\infty$. We will denote by $\varphi_k$ the corresponding unit $L^2$-norm eigenfunctions. Then, $\{\varphi_k\}_{k=1}^\infty$ is an orthonormal basis of $L^2(B)$. 

For any $k \in \nN$, $u(x,r) :=h_{\llg_k}(r) \varphi_k(x)$ solves the initial boundary value problem  \begin{eqnarray*}\left\{ \begin{array}{l} [\partial_r^2 +  \cot(r) \partial_r -\Delta] u(x,r) =0, (x,t) \in B \times (0,\pi),\\ u(y,r)=0,~ y \in S,\\  \partial_r u(x,0) = 0, ~ x \in B. \end{array}\right. \end{eqnarray*} Multiplying equation (\ref{E:Darbs}) by $u(x,r) \sin(r)$ and taking integration by parts, one obtains the {\bf orthogonality condition}:
\begin{eqnarray*} \int\limits_0^\pi \int\limits_S g(x,r) \partial_\nu \varphi_k(x,r) h_{\llg_k}(r) \sin(r) dt =0. \end{eqnarray*}
We assert that the support \& smoothness and orthogonality conditions are also sufficient for the range description of $\mR_S$:
\begin{theorem}\label{T:MainP}
Let $g$ be a function defined on $S \times [0,\pi]$. Then, there is a function $f \in C_0^\infty(\overline{B})$ such that $g=\mR_{\uS}(f)$ if and only if $g$ satisfies the following conditions: 
\begin{itemize}
\item[1.]{Smoothness $\&$ support condition:} $g \in C_0^\infty(S \times [0,2R])$.
\item[2.]{Orthogonality condition:} \begin{eqnarray}\label{E:OrthP} \int \limits_{0}^\pi \int\limits_{S} g(x,r) \partial_\nu \varphi_k(x) h_{\llg_k}(r) \sin(r) d\sg(x) dr =0.\end{eqnarray}
\end{itemize} 
\end{theorem}
 
Similarly to that of Theorem \ref{T:MainH}, the proof of this theorem also consists of three steps:
\begin{itemize}
\item[{\bf Step 1:}]
There is a unique solution $G \in C^\infty(\overline{B} \times [0,\pi])$ of the time-reversed equation \begin{eqnarray} \label{E:TimeRP}\left\{ \begin{array}{l} [\partial_r^2 + \cot(r) \partial_r -\Delta] G(x,r)=0,~(x,r) \in  B \times (0,2R], \\ G(x,r) =g(x,r),~ (x, r) \in S \times (0,2R],\\
G(x,2R)=0,~ G_r(x,2R)=0,~\ x \in B.
\end{array} \right.
\end{eqnarray}

\item[{\bf Step 2:}] The function $f^*=G(.,0)$ vanishes up to infinite order on $S$. 
\item[{\bf Step 3:}] Let $f$ be the zero extension of $f^*$ to $\uS^2$. Then, $g=\mR_S(f)$. 
\end{itemize}

The second and third steps are almost exactly the same as that of Theorem \ref{T:MainH}. One only needs to replace the hyperbolic trigonometric functions ($\sinh, \cosh, \tanh$, etc) by their usual analogs ($\sin,\cos,\tan$, etc). The proof of the first step (Theorem \ref{T:ST1P}, below) is a little bit different from that of the hyperbolic space (Theorem \ref{T:ST1}). The reason is that the Paley-Wiener theory for $\uS^2$ and $\Hs^n$ have different forms. We carry out here the first step, which is to prove: 

\begin{theorem}\label{T:ST1P} Assume the smoothness $\&$ support and orthogonality conditions. Then the problem (\ref{E:TimeRP}) has a unique solution $G \in C^\infty(\overline{B} \times [0,\pi])$.  \end{theorem}
 
 

The idea is to relate equation (\ref{E:TimeRP}) to a wave equation. Let us now introduce another Fourier-Legendre transform, which is essentially the spherical transform on $\uS^2$ (see, e.g., \cite{HelGeo}).



\begin{defi} Let $v \in C^\infty([-\pi,\pi])$ be an even function. Then the Fourier-Legendre transform of $m$ is $m \in \mathbb{Z} \mapsto \widehat{v}(m)$ defined by:
\begin{eqnarray*} \widehat{v}(m) = \frac{1}{2} \int\limits_0\limits^\pi v(r) h_m(r) \sin(r) dr. \end{eqnarray*}
The Fourier-Legendre series of $f$ is $$\sum\limits_{m=0}^\infty (2m+1) \widehat{v}(m) h_m(r).$$
\end{defi} 
The function $h_m$ is determined by equation (\ref{E:sphp}). The Fourier-Legendre series converges uniformly to $v$ if $v \in C^\infty[-\pi,\pi]$ (see, e.g., \cite{Olaf}). Moreover, one also has the following analog of the Paley-Wiener theorem:
\begin{theorem} (\cite{Olaf}) \label{T:PWP} An even function $v \in C^\infty([-\pi,\pi])$ is supported in $[-a,a]$ if and only if its Fourier-Legendre transform $m \mapsto \widehat{v}(m)$ extends to an entire function $g$ on $\cN$ satisfying: 
\begin{itemize} 
\item[i)] For any $k \in \nN$, there exists $C_k>0$ such that  $$| g(\llg) | \leq C_k (1+|\llg|)^{-k} e^{a \Im(\llg)},$$ 
\item[ii)] $g(\llg -\frac{1}{2})$ is an even function. 
\end{itemize}
The above extension $g$ of $\widehat{v}$ is unique. 

Conversely, any entire function $g$ on $\cN$ satisfying i) and ii) is the extension of $\widehat{v}$, for  a unique even function $v \in C^\infty([-\pi,\pi])$ supported in $[-a,a]$. 
\end{theorem}

Since $h_\llg$ depends on $\llg$ analytically, one can easily see that the function $g$ in the above theorem is defined by $$g(\llg) = \frac{1}{2} \int\limits_0\limits^\pi v(r) h_\llg(r) \sin(r) dr. $$ From now on, we will call $g$ the extended Fourier-Legendre transform of $v$, which is also denoted by $\widehat{v}$.  
Although the above theorem was stated for $v:[-\pi,\pi] \to \cN$, it is also true for $v:[-\pi,\pi] \to H$, for any Hilbert space $H$. 

\subsection{Proof of Theorem \ref{T:ST1P}}
As we mentioned, we will relate equation (\ref{E:TimeRP}) to a wave equation. For our convenience, we will denote by $\mC_a(H)$ the space of even functions $v \in C^\infty([-\pi,\pi]; H)$ such that $supp(f) \subset [-a,a]$. We also denote by $PW_a(H)$ the space of entire functions $g: \cN \to H$ such that for each $k \in \nN$, there is a constant $C_k$ such that $$\|g(\llg)\| \leq C (1+|\llg|)^{-N} e^{a \Im (\llg)},~\forall \llg \in \cN.$$

We now construct a bijective map $T$ between $\mC_{a}(H)$ and itself such that  \begin{equation}\label{E:T}T(u_{tt}+\frac{1}{4} u)= B_rT(u).\end{equation}
Here, $B$ is the Bessel-type operator $$B_r= \frac{d^2}{dr^2} +  \cot(r) \frac{d}{dr}. $$
For each $u \in \mC_a(H)$, let us consider its Fourier transform $\widetilde{u}$. Due to the standard Paley-Wiener theory for Fourier transform, $\widetilde{u}$ extends to an even function belonging to the class $PW_a$. Now, due to Theorem \ref{T:PWP}, there is a function $v \in \mC_a(H)$ such that its extended Fourier-Legendre transform $\widehat{v}(\llg)$ is equal to  $\widetilde{u}(\llg+\frac{1}{2})$. Defining $T(u)=v$, we now prove the identity (\ref{E:T}). 
Let $g_1$ be the Fourier transform of $u_{tt}+\frac{1}{4}u$. Standard calculations show $$g_1(\llg) = \left(-\llg^2+\frac{1}{4}\right) \widetilde{u}(\llg). $$
Let $g_2$ be the extended Fourier-Legendre transform of $B_rv$. We recall from (\ref{E:sphp}) that $$B_r h_\llg(r) = -\llg(\llg+1) h_\llg(r).$$ Standard integrations by parts give:  $$g_2(\llg) = -\llg (\llg+1)\widehat{v}(\llg).$$
We recall from the definition of $T$ that $v=T(u)$ implies $\widehat{v}(\llg) = \widetilde{u}(\llg+\frac{1}{2})$. Therefore, $g_2(\llg) = g_1(\llg+\frac{1}{2})$. Due to the definition of the operator $T$, we derive (\ref{E:T}).

Let us consider the wave equation  \begin{eqnarray}\label{E:IWP} \left\{\begin{array}{l} V_{tt}(x,t) +\frac{1}{4}V(x,t) - \Delta V(x,t) =0,~x \in B,~ t \geq -2R, \\ V(x,t)=b(x,t), ~ x \in S,\\ V(x,-2R)=0,~V_t(x,-2R)=0, ~ x \in B \end{array}  \right. \end{eqnarray}
where, $b=T^{-1}(g)$ belongs to $\mC_{2R}(H^s(S))$ since $g \in \mC_{2R}(H^s(S))$, for any $s>0$. We now prove that the above equation has a uniquely solution $V \in  \mC_{2R}(H^s(B))$. 

Standard theory for wave equations shows (\ref{E:IWP}) has a unique solution $V \in C^\infty([-2R,\infty); H^{s}(B))$. Extending the solution by zero for $t \leq -2R$, we obtain $V \in C^\infty(\rN; H^s(B))$.  We now prove that $V \in \mC_{2R}(H^s(B))$. Let $W(x,t)=V(x,t)-E(b)(x,t)$, where $E(b)(.,t)$ is the harmonic extension of $b(.,t)$ to $\overline{B}$. We arrive to \begin{eqnarray}\label{E:WWaveP} \left\{\begin{array}{l} W_{tt}(x,t) +\frac{1}{4}W(x,t) - \Delta W(x,t) =-P(x,t), (x,t) \in B \times \rN, \\ W(x,t)=0,~ x \in S, \\ W(x,-2R)=W_t(x,-2R)=0,~ x \in B. \end{array}  \right. \end{eqnarray} Here, $$P(x,t) = \partial_t^2 E(b)(x,t) + \frac{1}{4}E(b)(x,t) = E(b_{tt} + \frac{1}{4} b)(x,t)$$ is the harmonic extension of $b_{tt}+ \frac{1}{4}b$. 
Let $\{\varphi_k\}_k$ be the orthonormal basis of $L^2(B)$ defined above.  We expand the functions $W$ and $P$ in terms of $\varphi_k$:
\begin{eqnarray*}W(x,t) =\sum_{k} \og_k(t) \varphi_{k}(x), \quad P(x,t) =  \sum_{k} p_k(t) \varphi_{k}(x).\end{eqnarray*} 
Equation (\ref{E:WWaveP}) reduces to  \begin{eqnarray}\label{E:SODEP} \left\{ \begin{array}{l}  \og''_k(t)+\frac{1}{4}\og_k(t) + \llg_k (\llg_k+1) \og_k(t) = -p_k(t), \\ \og_k(-2R)=\og'_k(-2R)=0. \end{array} \right.\end{eqnarray} Here, $-\llg_k(\llg_k+1)$ is the eigenvalue corresponding to $\varphi_k$.
The function $p_k$ is determined by $$p_k(t) = \int\limits_B P(x,t) \varphi_k(x)=\frac{-1}{\llg_k(\llg_k+1)} \int \limits_B P(x,t) \Delta \varphi_k(x) dx .$$
Taking integration by parts, we obtain \begin{eqnarray*} p_k(t) &=& \frac{-1}{\llg_k(\llg_k+1)}\int\limits_S P(x,t) \partial_\nu \varphi_k(x) d\sg(x).\end{eqnarray*}
Since $P(x,t)$ is the (harmonic) extension of $b_{tt}(x,t)+\frac{1}{4}b(x,t)$:  \begin{eqnarray}\label{E:pk} p_k(t) = \frac{-1}{\llg_k(\llg_k+1)}\int\limits_S \left[b_{tt}(x,t) + \frac{1}{4} b(x,t) \right] \partial_\nu \varphi_k(x) d\sg(x).\end{eqnarray}
The Fourier transform $\widetilde{p}$ of $p$ is determined by 
$$\widetilde{p}_k(\llg) = \frac{\llg^2-\frac{1}{4}}{\llg_k(\llg_k+1)} \int\limits_S \widetilde{b}(x, \llg) \partial_\nu \varphi_k(x) d\sg(x)= \frac{\llg^2-\frac{1}{4}}{\llg_k(\llg_k+1)} \int\limits_S \widehat{g}(x,\llg-\frac{1}{2}) \partial_\nu \varphi_k(x) d\sg(x).$$
The last equality is due to the relation $g=T(b)$. 
We notice that the orthogonality condition (\ref{E:OrthP}) can be read as $$\int\limits_S \widehat{g}(x, \llg_k) \partial_\nu \varphi_k(x) d\sg(x)=0,\quad \forall k \in \nN.$$
Therefore,  $\widetilde{p}(\llg_k+\frac{1}{2})= 0$. 

Let $\gamma=\gamma(t)$ be such that its Fourier transform $\widetilde{\gamma}$ satisfies: $$\left[-\llg^2 + \frac{1}{4}+ \llg_k(\llg_k+1)\right]\widetilde{\gamma}(\llg) = - \widetilde{p}_k(\llg).$$ 
Since $p_k \in \mC_r(\rN)$, $\widetilde{p}_k$ is an even function belonging to $PW_{2R}$. Recalling that $\widetilde{p}(\llg_k+\frac{1}{2})=0$,  one obtains $$\widetilde{\gamma}(\llg) =\frac{-\widetilde{p}_k(\llg)}{-\llg^2+\frac{1}{4} +\llg_k(\llg_k+1)}= \frac{\widetilde{p}_k(\llg)}{\llg^2-(\llg_k+\frac{1}{2})^2}$$ is also an even function belonging to $PW_{2R}$. Due to the standard Paley-Wiener theory, $\gamma \in \mC_{2R}(\rN)$. We can verify that $\ag$ solves equation (\ref{E:SODEP}). Hence, $\og_k=\gamma$, due to the uniqueness of the solution of the equation (\ref{E:SODEP}). This implies that $\og_k \in C_{2R}(\rN)$ for all $k$. Since $W \in C^\infty(\rN; H^s(B))$, we obtain $W \in \mC_{2R}(H^s(B))$. We then conclude that the solution  $V=W+ E(b)$ of (\ref{E:TimeRP}) satisfies $V \in \mC_{2R}(H^s(B))$. 

Now, let $G=T(V) \in \mC_{2R}(H^s(B))$. Due to (\ref{E:T}) and (\ref{E:IWP}), we obtain the equation (\ref{E:TimeRP}):
\begin{eqnarray*} \left\{ \begin{array}{l} [\partial_r^2 + \cot(r) \partial_r -\Delta] G(x,r)=0,~(x,r) \in  B \times (0,2R], \\ G(x,r) =g(x,r),~ (x, r) \in S \times (0,2R],\\
G(x,2R)=0,~ G_r(x,2R)=0,~\ x \in B.
\end{array} \right.
\end{eqnarray*}  Since $s$ is arbitrary, we conclude that $G \in C^\infty(\overline{B} \times [0,\pi])$. This concludes the theorem. 
\section*{Appendix}
\subsection{Proof of Proposition \ref{P:Com}}

\begin{proof}
Let  \begin{eqnarray*}
A&=& d_s \mD_k = d_s \left[d_s^2 +(n-1) \coth(s) d_s - \frac{k(k+n-2)}{\sinh^2 s}\right], \\ 
B&=&(n+k-2)\coth(s) \mD_k \\ &=& (n+k-2)\coth(s) \left[d_s^2 +(n-1) \coth(s) d_s - \frac{k(k+n-2)}{\sinh^2 s}\right],\\
C&=&\mD_{k-1}d_s \\ &=& \left[d_s^2 +(n-1) \coth(s) d_s - \frac{(k-1)(k+n-3)}{\sinh^2 s}\right] d_s, \\ 
D&=& \mD_{k-1} (n+k-2) \coth(s) \\&=&(n+k-2) \left[\partial_s^2 +(n-1) \coth(s) d_s - \frac{(k-1)(k+n-3)}{\sinh^2 s}\right]  \coth(s).
\end{eqnarray*}

We have \begin{eqnarray*}
A-C&=& d_s \left[d_s^2 +(n-1) \coth(s) d_s - \frac{k(k+n-2)}{\sinh^2 s}\right], \\ 
&-& \left[d_s^2 +(n-1) \coth(s) d_s - \frac{(k-1)(k+n-3)}{\sinh^2 s}\right] d_s \\ 
&=& -(n-1) \frac{1}{\sinh^2 s} d_s +2k(k+n-2)\sinh^{-3}(s) \cosh(s) d_s\\ &+& \left[-k(k+n-2)+(k-1)(k+n-3)\right]\frac{1}{\sinh^2 s} \\
&=& -2(n+k-2)\frac{1}{\sinh^2 s}d_s +2k(k+n-2)\sinh^{-3}(s) \cosh(s).
\end{eqnarray*}
On the other hand, 
\begin{eqnarray*}
D-B&=& (n+k-2) \left[d_s^2 +(n-1) \coth(s) d_s - \frac{(k-1)(k+n-3)}{\sinh^2 s}\right]  \coth(s)\\
&-& (n+k-2)\coth(s) \left[d_s^2 +(n-1) \coth(s) d_s - \frac{k(k+n-2)}{\sinh^2 s}\right].\\
&=& (n+k-2)\left[d^2 \coth(s) + 2 \partial_s \coth(s) d_s +(n-1) \coth(s) d_s \coth(s) \right] \\
&+&(n+k-2)\left[k(n+k-2) -(k-1)(n+k-3) \right]\sinh^{-2}(s)\coth(s) \\ &=& (n+k-2)\left[-2 \sinh^{-2}(s)d_s+ 2k \sinh^{-3}(s) \cosh(s) \right]. 
\end{eqnarray*}
We conclude that $$A-C=D-B.$$
This proves the proposition.
\end{proof}

\subsection{Proof of Proposition \ref{P:D_m}} 
\begin{proof}
We have \begin{eqnarray*} d_s[ \cosh^i(s) \sinh^l(s)] &=& i \cosh^{i-1}(s) \sinh^{l+1}(s) +l \cosh^{i+1}(s) \sinh^{l-1}(s) \\ &=&(i+l) \cosh^{i-1}(s) \sinh^{l+1}+ l \cosh^{i-1}\sinh^{l-1} .\end{eqnarray*}
Hence,
 \begin{eqnarray*}(n-1) \coth(s) d_s[ \cosh^i(s) \sinh^l(s)] &=& (n-1)(i+l) \cosh^{i}(s) \sinh^{l} \\ &+& (n-1) l \cosh^i (s)\sinh^{l-2}(s) .\end{eqnarray*}

\begin{eqnarray*} d^2_s[ \cosh^i(s) \sinh^l(s)] &=& d_s^2 (\cosh^i(s)) \sinh^l(s) + 2 d_s \cosh^i(s) d_s \sinh^l(s) \\ &+& \cosh^i(s) d^2_s\sinh^l(s) \\ &=& [i^2 \cosh^i(s) -i(i-1)\cosh^{i-2}(s)]\sinh^l(s) \\ &+& 2il \cosh^i \sinh^l(s) \\ &+& \cosh^i(s) [l^2 \sinh^l(s) + l(l-1) \sinh^{l-2}(s)] \\ &=& (i+l)^2 \cosh^i(s) \sinh^l(s) - i(i-1) \cosh^{i-2}(s) \sinh^l(s) \\ &+& l(l-1) \cosh^i(s) \sinh^{l-2}(s).\end{eqnarray*}

Therefore, \begin{eqnarray*}
&&[d_s^2+(n-1) \coth(s) d_s](\cosh^i(s)\sinh^l(s)) \\ &=& (i+l + n-1) (i+l) \cosh^i(s) \sinh^l(s) \\ &-&i(i-1) \cosh^{i-2}(s) \cosh^{l}(s) +(n-2+l)l \cosh^i(s) \sinh^{l-2}(s)
\end{eqnarray*}
Since $u_i = \cosh^i(s) \sinh^{-n-m+2}$, we obtain \begin{eqnarray*}
\mD_m u_i(s) &=& \left[d_s^2+(n-1) \coth(s) d_s-\frac{m(m+n-2)}{\sinh^2(s)} \right]u_i(s) \\ &=& (m-i-1) (m+n-2-i) \cosh^i(s) \sinh^{-m-n+2}(s) \\ &-&i(i-1) \cosh^{i-2}(s) \cosh^{-m-n+2}(s) \\ &=& \kappa_i u_i(s) -i(i-1)u_{i-2}.\end{eqnarray*}
\end{proof}
\section*{Acknowledgement} The author is thankful to Professors M. Agranovsky and P. Kuchment for helpful discussions.  This work was started during the author's stay at MSRI as a Postdoctoral Research Fellow under the program Inverse Problems and Applications. The author thanks MSRI and the program organizers, especially Professor G. Uhlmann, for the supports. 


\begin{thebibliography}{AKQ07}

\bibitem[ABK96]{ABK}
Mark Agranovsky, Carlos Berenstein, and Peter Kuchment, \emph{Approximation by
  spherical waves in {$L\sp p$}-spaces}, J. Geom. Anal. \textbf{6} (1996),
  no.~3, 365--383 (1997). \MR{1471897 (99c:41038)}

\bibitem[AFK09]{AFK}
Mark Agranovsky, David Finch, and Peter Kuchment, \emph{Range condition for a
  spherical mean transform}, Inverse Problems and Imaging \textbf{3} (2009),
  no.~3, 373--382.

\bibitem[AK06]{AmK06}
Gaik Ambartsoumian and Peter Kuchment, \emph{A range description for the planar
  circular {R}adon transform}, SIAM J. Math. Anal. \textbf{38} (2006), no.~2,
  681--692 (electronic). \MR{MR2237167 (2007e:44005)}

\bibitem[AKQ07]{AKQ}
Mark Agranovsky, Peter Kuchment, and Eric~Todd Quinto, \emph{Range descriptions
  for the spherical mean {R}adon transform}, J. Funct. Anal. \textbf{248}
  (2007), no.~2, 344--386. \MR{MR2335579}

\bibitem[AN10]{ANg}
Mark Agranovsky and Linh~V. Nguyen, \emph{Range conditions for a spherical mean
  transform and global extendibility of solutions of the {D}arboux equation},
  J. Anal. Math. \textbf{112} (2010), 351--367. \MR{2763005}

\bibitem[BZ80]{BerZalc}
Carlos~A. Berenstein and Lawrence Zalcman, \emph{Pompeiu's problem on symmetric
  spaces}, Comment. Math. Helv. \textbf{55} (1980), no.~4, 593--621. \MR{604716
  (83d:43012)}

\bibitem[EK93]{Ep}
Charles~L. Epstein and Bruce Kleiner, \emph{Spherical means in annular
  regions}, Comm. Pure Appl. Math. \textbf{46} (1993), no.~3, 441--451.
  \MR{MR1202964 (93m:43005)}

\bibitem[Eva98]{Evb}
Lawrence~C. Evans, \emph{Partial differential equations}, Graduate Studies in
  Mathematics, vol.~19, American Mathematical Society, Providence, RI, 1998.
  \MR{MR1625845 (99e:35001)}

\bibitem[FHR07]{FHR}
David Finch, Markus Haltmeier, and Rakesh, \emph{Inversion of spherical means
  and the wave equation in even dimensions}, SIAM J. Appl. Math. \textbf{68}
  (2007), no.~2, 392--412. \MR{MR2366991 (2008k:35494)}

\bibitem[FPR04]{FPR}
David Finch, Sarah~K. Patch, and Rakesh, \emph{Determining a function from its
  mean values over a family of spheres}, SIAM J. Math. Anal. \textbf{35}
  (2004), no.~5, 1213--1240 (electronic). \MR{MR2050199 (2005b:35290)}

\bibitem[FR07]{FR07}
David Finch and Rakesh, \emph{The spherical mean value operator with centers on
  a sphere}, Inverse Problems \textbf{23} (2007), no.~6, S37--S49.
  \MR{MR2440997}

\bibitem[Hel70]{HelDual}
Sigur{\dbar}ur Helgason, \emph{A duality for symmetric spaces with applications
  to group representations}, Advances in Math. \textbf{5} (1970), 1--154
  (1970). \MR{0263988 (41 \#8587)}

\bibitem[Hel84]{HelGeo}
Sigurdur Helgason, \emph{Groups and geometric analysis}, Pure and Applied
  Mathematics, vol. 113, Academic Press Inc., Orlando, FL, 1984, Integral
  geometry, invariant differential operators, and spherical functions.
  \MR{754767 (86c:22017)}

\bibitem[KK08]{KKun}
Peter Kuchment and Leonid Kunyansky, \emph{Mathematics of thermoacoustic
  tomography}, European J. Appl. Math. \textbf{19} (2008), no.~2, 191--224.
  \MR{MR2400720}

\bibitem[Ole44]{Olev}
M.~Olevsky, \emph{Quelques th\'eor\`emes de la moyenne dans les espaces \`a
  courbure constante}, C. R. (Doklady) Acad. Sci. URSS (N.S.) \textbf{45}
  (1944), 95--98. \MR{0011886 (6,230c)}

\bibitem[{\'O}S08]{Olaf}
Gestur {\'O}lafsson and Henrik Schlichtkrull, \emph{A local {P}aley-{W}iener
  theorem for compact symmetric spaces}, Adv. Math. \textbf{218} (2008), no.~1,
  202--215. \MR{2409413 (2010c:43019)}

\bibitem[Vol03]{Volch-b}
V.~V. Volchkov, \emph{Integral geometry and convolution equations}, Kluwer
  Academic Publishers, Dordrecht, 2003. \MR{2016409 (2005e:28001)}

\end{thebibliography}

\def\dbar{\leavevmode\hbox to 0pt{\hskip.2ex \accent"16\hss}d}
\providecommand{\bysame}{\leavevmode\hbox to3em{\hrulefill}\thinspace}
\providecommand{\MR}{\relax\ifhmode\unskip\space\fi MR }
\providecommand{\MRhref}[2]{%
  \href{http://www.ams.org/mathscinet-getitem?mr=#1}{#2}
}
\providecommand{\href}[2]{#2}

\end{document}